\documentclass{amsart}

\usepackage{amssymb}
\usepackage{graphicx}
\usepackage{cite}
\usepackage{hyperref}

\newtheorem{theorem}{Theorem}[section]
\newtheorem{lemma}[theorem]{Lemma}
\newtheorem{proposition}[theorem]{Proposition}
\newtheorem{corollary}[theorem]{Corollary}
\theoremstyle{definition}

\theoremstyle{remark}
\newtheorem{remark}[theorem]{Remark}
\numberwithin{equation}{section}
\newtheorem{cproof}{Computer assisted proof}

\newcommand{\R}{{\mathbb R}}

\newcommand{\C}{\mathcal{C}}

\newcommand{\eqn}[2]{\begin{equation}\begin{split} #1\end{split}\label{#2}\end{equation}}
\newcommand{\eq}[1]{\begin{equation}\begin{split} #1\end{split}\notag\end{equation}}
\newcommand{\N}{\mathbb{N}}

\author{Mark Allen}
\address{Department of Mathematics, Brigham Young University, Provo,
  UT 84602}
\email{allen@mathematics.byu.edu}

\author{Blake Barker}
\address{Department of Mathematics, Brigham Young University, Provo,
  UT 84602}
\email{blake@mathematics.byu.edu}

\author{Jason Gardiner}
\address{Department of Mathematics, Brigham Young University, Provo,
  UT 84602}
\email{jasongardiner@mathematics.byu.edu}

\author{Mingyan Zhao}
\address{Department of Mathematics, Brigham Young University, Provo,
  UT 84602}
\email{zmy407737382@gmail.com}

\thanks{M.~Allen was supported by Simons Foundation Collaboration Grant Award Number 637757}

\title[Free Boundary Problem]{Minimizers of a free boundary problem\\ on three-dimensional cones}

\begin{document}

\maketitle
 \begin{abstract}
  We consider a free boundary problem on three-dimensional cones depending on a parameter $c$ and study when the free boundary 
  is allowed to pass through the vertex of the cone. Combining analysis and computer-assisted proof, we show that when $c \leq 0.43$ the free boundary 
  may pass through the vertex of the cone. 
 \end{abstract}

\section{Introduction}
 We study solutions to the one-phase free boundary problem 
    \[
    \begin{aligned}
      \Delta u &=0  &\text{ in } &\{u>0\}, \\
         |\nabla u|&=1  &\text{ on } &\partial \{u>0\},
   \end{aligned}
  \]
 on right circular cones in $\R^4$. We are interested in determining when the free boundary $\partial \{u>0\}$ is allowed to pass through the vertex of the cone. 
 
 The one-phase free boundary problem has been well-studied since the seminal paper \cite{ac81}, and has many applications including two dimensional flow problems as well as heat flow problems, see \cite{cs05}. The literature is too extensive to list here; however, we do mention recent results regarding the variable coefficient situation which pertains to our problem.  
 When considering the applications on a manifold, one studies a variable coefficient problem in divergence form: 
  \begin{equation}    \label{e:variphase}
    \begin{aligned} 
      \partial_j(a^{ij}(x)u_i) &=0  &\text{ in } &\{u>0\} \\
         a^{ij}(x)u_i u_j &=1  &\text{ on } &\partial \{u>0\}.
   \end{aligned}\
  \end{equation}
 Solutions of \eqref{e:variphase} may be found inside a bounded domain $\Omega$ by minimizing the functional:
  \begin{equation}   \label{e:functional2}
   \int_{\Omega} a^{ij}(x)v_iv_j + \chi_{\{v>0\}}. 
  \end{equation}
  Since the functional is not convex, minimizers of \eqref{e:functional2} may not be unique and there exist solutions to \eqref{e:variphase} 
  which are not minimizers of \eqref{e:functional2}. 
  When the coefficients $a^{ij}(x)$ are H\"older continuous, the regularity of the free boundary $\partial \{u>0\}$ was accomplished in \cite{dfs16}. For coefficients $a^{ij}(x)$ assumed merely to be bounded, measurable, and satisfying the usual
  ellipticity conditions, regularity of the solution and its growth away from the free boundary was studied in \cite{pt16}. However, to date nothing is know regarding the regularity
  of the free boundary when the coefficients $a^{ij}(x)$ are allowed to be discontinuous. In this paper we are interested in how the free boundary interacts with 
  isolated discontinuous points of the coefficients $a^{ij}(x)$. 
  In the context of a hypersurface, these points are considered to be a topological singularity. The simplest such case is the vertex of a cone. 
  
  In this paper we improve the known results on when the free boundary may pass through the vertex of a three-dimensional right circular cone. 
  We consider minimizers of the functional 
   \begin{equation}   \label{e:cfunctional}
   J_c(w)=\int_{\C_1} |\nabla w|^2 + \chi_{\{w>0\}},
  \end{equation}
  with fixed boundary data $w=\zeta$ on $\partial C_1$ where 
  \[
  \C_1 := \{x \in \mathbb{R}^4  \mid x_4 = c\sqrt{x_1^2+x_2^2 + x_3^2} \  \text{ and } \ x_1^2+x_2^2 + x_3^2 <1 \},
  \]
   the parameter $c \geq 0$, and $\nabla_c$ is the inherited gradient on the cone. 
  In \cite{a17} the first author showed there exists $c_0>0$ such that if $c>c_0$ and $u$ is a minimizer of \eqref{e:cfunctional}, then $0 \notin \partial \{u>0\}$. The value
  $c_0$ arises as a critical case for an inequality involving Legendre functions (see Section \ref{s:critical}). The question of the free boundary passing through the cone was reduced to the study of a particular homogeneous solution $\tilde{u}_c$ (see Section \ref{s:prelim}). When $c> c_0$, the solution $\tilde{u}_c$ is not stable, and therefore
  cannot be a minimizer for the functional. However, when $c<c_0$, the solution $\tilde{u}_c$ is stable. This suggests that $\tilde{u}_c$ could be a minimizer for the functional
  for $c \leq c_0$, and therefore the free boundary can pass through the vertex of the cone. 
  
  The first author also showed in \cite{a17} that  there exists some arbitrarily small $c_1>0$ (with $c_1 < c_0$) such that if $c\leq c_1$, then there exists a minimizer $u$ to \eqref{e:cfunctional} with $0\in \partial \{u>0\}$. That $c_1>0$ is significant precisely because this result holds only for right circular cones when the dimension of the cone is greater than 2. For $2$-dimensional right circular cones, the first author and H. Chang-Lara showed in \cite{acl15} that the free boundary never passes through the vertex for any $c>0$.  
  
  The result in \cite{a17} naturally raises the question of what occurs for $c_1 < c \leq c_0$.  The techniques in \cite{a17} were purely analytic and used a compactness argument by letting $c \to 0$. The methods and proofs heavily relied on known results when the cone is flat. In this paper we first numerically calculate the value $c_0 \approx 0.5884$. Combining improved analytic techniques with computer-assisted proof, we show that for $c \leq 0.43$, 
   the free boundary does indeed pass through the vertex of the cone. This is our main result and stated in Theorem \ref{t:main}. 
  
  We also remark that there is a well known connection between solutions to the one-phase problem and minimal surfaces. 
  Theorem \ref{t:main} is a one-phase analogue of a theorem of Morgan \cite{m02} for minimal surfaces on cones, see discussion in \cite{a17}.

      \subsection{Outline and Notation} 
          The outline of this paper is as follows. In Section \ref{s:prelim} we define the notion of a solution to the free boundary problem and the corresponding functional. We also recall the 
          Legendre functions which will be important to our problem. 
           In Section \ref{s:critical} we calculate numerically the critical value $c_0$.   
           In Section \ref{s:subsolutions} we construct a continuous family of subsolutions to our free boundary problem, and in Section \ref{s:supersolution} we construct a continuous 
           family of supersolutions. Using the results from Sections \ref{s:subsolutions} and \ref{s:supersolution} we are able to prove Theorem \ref{t:main}.  
          
          Throughout the paper we will consider functions depending on radius $r$ and angle $\phi$. For geometric reasons it is often easier to consider functions of these variables. 
          However, for numerical computations it is more advantageous to use the change of variable $t=\cos \phi$.  In order to distinguish between 
          when a function is a variable of $\phi$ or $t$, we will use the convention $u(r,\cos \phi)=\tilde{u}(r,\phi)$. 
          We will use the following notation throughout the paper. 
          \begin{itemize}
           \item $\C$ is always a cone of type $\{x \in \mathbb{R}^4 \mid x_4 = c \sqrt{x_1^2+x_2^2 + x_3^2}\}$.
           \item $\C_r := \{(y_1,y_2,y_3,y_4) \in \C : \sqrt{y_1^2 +y_2^2 + y_3^2} <r\}$. 
           \item $c$ is always the constant appearing in the definition of the cone in $\C$. 
           \item $\nabla_c$ refers to the gradient on $\C$ arising from the inherited metric as explained in Section \ref{s:prelim}. 
           \item $\Delta_c$ refers to the Laplace-Beltrami operator on $\C$ as explained in \ref{s:prelim}.  
          \item We say that a function $u$ is $c$-harmonic if $\Delta_c u=0$. 
          \end{itemize}

\section{Preliminaries}  \label{s:prelim}
   In this section we recall the results and notation from \cite{a17}.  Minimizers of \eqref{e:cfunctional} are solutions to     
   \begin{equation}   \label{e:phase1}
     \begin{cases}
      u \geq 0 \\
      u\text{ is continuous in  } \C_1 \\
           \Delta_c u =0  \text{ in } \{u>0\} \\
      \partial\{u>0\} \setminus \{0\} \text{ is locally smooth} \\
         |\nabla_c u|=1  \text{ on } \partial \{u>0\}\setminus \{0\},
   \end{cases}
   \end{equation}  
  where $\nabla_c$ is the gradient on $\C$ from the inherited metric, and 
  $\Delta_c$ is the Laplace-Beltrami on $\C$.

    We consider the spherical parametrization of the cone $\C$. Using spherical coordinates we have 
  \begin{equation}  \label{e:spherical}
   y_1 = r \cos \theta \sin \phi, \qquad
    y_2 = r \sin \theta \sin \phi, \qquad
    y_3 = r \cos \phi, \qquad
    y_4 = cr.
  \end{equation}

  Under these coordinates the area form is $\sqrt{\text{det}(g)}=r^2 \sin \phi \sqrt{1+c^2}$. 
  The local coordinates $g^{ij}$ are given by 
\[ 
 \left( 
        \begin{array}{ccc}
          r^{-2} \sin^{-2} \phi  & 0       & 0 \\
          0                             & r^{-2} & 0 \\
          0                             & 0       & (1+c^2)^{-1} 
        \end{array} 
  \right),
\] 
and in these local coordinates we minimize
 \[
  \int_{B_1} \sqrt{g} g^{ij} u_i u_j + \sqrt{g} \chi_{\{u>0\}}. 
 \]
 Any minimizer will satisfy 
 \[
  \begin{cases}
    \frac{1}{\sqrt{g}} \partial_j (\sqrt{g} g^{ij} u_i) &=0  \quad\text{ in }  \{u>0\} \\
   \sqrt{g} g^{ij} u_iu_j &= \sqrt{g}   \quad\text{ on } \partial\{u>0\}\setminus \{0\},
  \end{cases}
 \]
and so a minimizer of \eqref{e:cfunctional} is a solution to \eqref{e:phase1}. We note that the first condition is written out as 
  \begin{equation}  \label{e:beltrami}
  \frac{1}{1+c^2} \left(u_{rr} + \frac{2}{r}u_r  \right) + \frac{1}{r^2}
   \left(\frac{u_{\theta \theta}}{\sin \phi} + \frac{\cos \phi}{\sin \phi} u_{\phi} + u_{\phi \phi} \right) =0
  \end{equation}
  in the set $\{u>0\}$.
   
  In this paper, we will often consider solutions which only depend on the variables $r,\phi$. In this situation we have the following formula for $|\nabla_c u|^2$
 where $\nabla_c$ is the gradient on the cone $\C$. 
 \begin{equation}   \label{e}
  |\nabla_c v|^2 = \frac{1}{1+c^2} v_r^2 + \frac{1}{r^2} v_{\phi}^2. 
 \end{equation}

 As shown in \cite{a17}, in order to classify whether the free boundary passes through the vertex, it is sufficient to determine whether or not one-homogeneous solutions are minimizers. 
   A particular candidate solution is $\tilde{u}_c(r,\theta,\phi)=\max\{r \tilde{f}_c (\phi),0\}$  with $\tilde{f}_c(\phi)$ solving
   \[
   \begin{cases}
    &\tilde{f}_c''(\phi) + \frac{\cos \phi}{\sin \phi} \tilde{f}_c'(\phi) + \frac{2}{1+c^2} \tilde{f}_c(\phi)=0 \text{ for } 0\leq \phi < \pi, \\
    &\tilde{f}_c'(\phi_c)=1,
    \end{cases}
   \]  
   where $\phi_c$ is the unique value such that $\tilde{f}_c(\phi_c)=0$. 
   In \cite{a17} it is shown that if $\C$ is a three-dimensional right circular cone, and if $u$ is a $1$-homogeneous minimizer, then up to rotation $u=\tilde{u}_c$. Furthermore, 
   if the free boundary passes through the vertex, then after performing a blow-up, one obtains that $\tilde{u}_c$ is a minimizer. Thus, the problem of determining when the free 
   boundary of a minimizer may pass through the vertex is completely reduced to determining when $\tilde{u}_c$ is a minimizer.

    If $\Delta_c r^{\alpha} \tilde{f}(\theta,\phi)=0$ and independent of $\theta$, then 
    \[
     \tilde{f}''(\phi) + \frac{\cos \phi}{\sin \phi} \tilde{f}'(\phi) + \frac{\alpha(\alpha+1)}{1+c^2} \tilde{f}(\phi)=0. 
    \]
   Under the change of variable $t=\cos \phi$ with $f(t)=\tilde{f}(\phi)$, this equation becomes 
   \begin{equation}  \label{e:legalpha}
    (1-t^2)f''(t) -2tf'(t) + \frac{\alpha(\alpha+1)}{1+c^2}f(t)=0,
   \end{equation}
   which is a Legendre equation and well-studied. Most relevant to us are the values $\alpha=1$ and $\alpha=-1/2$. For convenience we write the equation as 
   \begin{equation}   \label{e:leg1}
      (1-t^2)f''(t) -2tf'(t) + \beta f(t)=0 \text{ for } -1<t<1. 
   \end{equation}
   There exist two linearly independent solutions, and we are interested in the Legendre function of the first kind. This solution may be written as a power series 
   \[
    f(t,\beta):= \sum_{n=0}^{\infty} a_n(\beta)(1-t)^n,
   \]
    with the recursion relation
   \[
    a_{n+1}(\beta) = \frac{a_n}{2} \left(\frac{n^2+n-\beta}{(n+1)^2} \right).
   \]
   Using the estimates in Proposition \ref{p:e1} below, it is clear that for $\beta \in [-2,2]$ and $t \in (-1,3)$ the series will converge absolutely. 
   When $\beta=2/(1+c^2)$ we denote the solution as $f_c(t)$, and when $\beta=-(1/4)/(1+c^2)$ we denote the solution as $g_c(t)$.

   We have the following elementary estimate for error bounds. We will only be interested in $t \in (-1,1]$, however, similar bounds hold for $t \in (1,3)$. 
   \begin{proposition}  \label{p:e1}
    If $f_n(t,\beta):=\sum_{k=0}^n a_k(\beta)(1-t)^k$, then for $t \in (-1,1]$ and $\beta\in[-2,2]$,
    \[
     |f(t,\beta)-f_n(t,\beta)| \leq \frac{4|a_0|}{t+1}  \left(\frac{1-t}{2}\right)^{n+1}. 
    \]
    For $j$ derivatives we also have 
    \[
     |f^{(j)}(t,\beta)-f^{(j)}_n(t,\beta)| \leq 2|a_0| \left|\left(\frac{((1-t)/2)^{n+1}}{1- (1-t)/2}\right)^{(j)} \right|.
    \]
   \end{proposition}

   \begin{proof}
    For $\beta \in [-2,2]$ we have that $|a_n| \leq a_0 (1/2)^{n-1}$. Thus, 
    \[
    \begin{aligned}
      |f(t,\beta)-f_n(t,\beta)| &\leq \left| \sum_{k=n+1}^{\infty} |a_k| (1-t)^k\right| \\
                                         &\leq 2|a_0|\sum_{k=n+1}^{\infty}  |(1-t)/2|^k \\
                                         &= 2|a_0|\frac{|(1-t)/2|^{n+1}}{1- |(1-t)/2|} \\
                                         &= 4|a_0|  \frac{|(1-t)/2|^{n+1}}{t+1}. 
      \end{aligned}
    \]
     We also have 
    \[
    \begin{aligned}
      |f'(t,\beta)-f'_n(t,\beta)| &= \left|\sum_{k=n+1}^{\infty}  k a_k (1-t)^{k-1}\right| \\
                                           &\leq 2|a_0|\left|\sum_{k=n+1}^{\infty}  k((1-t)/2)^{k-1}\right|  \\
                                           &= 2|a_0| \left|\left(\frac{((1-t)/2)^{n+1}}{1- (1-t)/2}\right)' \right|\\
                                           &= 2^{1-n}|a_0|\left|\left( \frac{(1-t)^{n+1}}{1+t}\right)'\right|\\
                                           &= 2^{1-n}|a_0| \frac{(1-t)^n(2+n(1+t))}{(1+t)^2}.
      \end{aligned}
    \]
    In a similar manner we have 
    \[
        \begin{aligned}
     |f^{(j)}(t,\beta)-f^{j}_n(t,\beta)| &\leq 2|a_0| \left|\left(\frac{|(1-t)/2|^{n+1}}{1- (1-t)/2}\right)^{(j)} \right|\\
     &=  2^{1-n}|a_0|\left|\left( \frac{(1-t)^{n+1}}{1+t}\right)^{(j)}\right|.
           \end{aligned}
    \]
   \end{proof}
   
  \subsection{Computer assisted proof} We will use computer assisted proof, also referred to as rigorous computation, to verify several results of the paper.  Researchers have used rigorous computation to solve several conjectures, including Mitchell Feigenbaum's universality conjecture in non-linear dynamics, the Kepler conjecture, and the 14th of Smale's problems. We say that a computation is rigorous if all possible errors are accounted for and tracked so that the resulting answer is an interval that is guaranteed to contain the actual desired quantity. In order to bound machine truncation error, we use the interval arithmetic package INTLAB \cite{R}.
        
  \subsection{Analytic interpolation} At times, we will interpolate the functions $f$ and $g$, with interpolation error bounds, as part of our rigorous computation to verify results. We may obtain a relatively low degree polynomial approximation of an analytic function of interest when interpolating, with error bounds on the order of machine epsilon. Indeed, the error decays exponentially with respect to the number of interpolation nodes. Further, one only need evaluate the modulus of the analytic function on an ellipse to obtain the bound. More than just speeding up evaluation of the function, interpolation reduces the so called wrapping effect that can occur when using interval arithmetic, which refers to the size of intervals growing unnecessarily large to the point of not being useful.  The details of how we interpolate are as given in \cite{Barker,Barker2014}, and the theory regarding error bounds are given in \cite{RW,Ta}.

   \section{The critical value}  \label{s:critical}
   Let $t_c$ be the unique value such that $f_c(t_c)=0$, and let $c_0$ be the unique value such that 
   \[
   \frac{|t_{c_0}|}{1-t_{c_0}^2} = \frac{|g'(t_{c_0})|}{g(t_{c_0})}. 
   \]
   As shown in \cite{a17}, by applying a second variational formula, $\tilde{u}_c$ is not stable and hence not a minimizer if 
    \[
     \frac{|t_c|}{1-t_c^2} > \frac{|g'(t_c)|}{g(t_c)}, 
    \]
    which is the case whenever $c>c_0$. If $c\leq c_0$, the above inequality fails. This implies that $\tilde{u}_c$ is a stable solution and suggests that $\tilde{u}_c$ is a minimizer whenever $0\leq c \leq c_0$. In this paper we use rigorous computation to show that there exists a $0< \tilde c_0< c_0$ such that $\tilde{u}_c$ is a minimizer whenever $0\leq c\leq \tilde c_0$. We numerically approximate $c_0 \approx 0.5884=: \tilde c_0$.
    
    To calculate this approximation of $c_0$, we interpolated the function $b(c) = t_c$ and then used a root finder to solve
    \[
     \frac{|b(c)|}{1-b(c)^2} - \frac{|g'(b(c))|}{g(b(c))} = 0 
    \]
    for $c$.


\section{A continuous family of subsolutions}  \label{s:subsolutions}

We briefly review how to show that $\tilde{u}_c$ is a minimizer. In order to do so we recall the definitions of a subsolution and supersolution to our free boundary problem. 
We say that $v$ is a subsolution to the free boundary problem if 
 \begin{equation}   \label{e:subs}
     \begin{cases}
      v \geq 0 \\
      v\text{ is continuous in  } \Omega \\
           \Delta_c v \geq 0  \text{ in } \{v>0\} \\
         \lim_{r \to 0} \sup_{B_r(x)} |\nabla_c v| > 1  \text{ for any } x \in \partial \{v>0\}\setminus \{0\}.
   \end{cases}
   \end{equation}  
   Similarly, $v$ is a supersolution to the free boundary if 
    \begin{equation}   \label{e:sups}
     \begin{cases}
      v \geq 0 \\
      v\text{ is continuous in  } \Omega \\
           \Delta_c v \leq 0  \text{ in } \{v>0\} \\
         \lim_{r \to 0} \sup_{B_r(x)} |\nabla_c v| < 1  \text{ for any } x \in \partial \{v>0\}\setminus \{0\}.
   \end{cases}
   \end{equation}  
   
   From the comparison principle as well as the gradient condition, it follows that if $u$ is a solution and $v$ a subsolution to the free boundary problem on $B_r$
   with $v \leq u$ in $B_r$, then $v<u$ in $\{u>0\} \cap B_r$ and also $\partial \{u>0\} \cap \partial \{v>0\} = \emptyset$ in $B_r$. An analogous statement holds for supersolutions. 
   To prove that $\tilde{u}_c$ is a minimizer we construct a continuous family of subsolutions $\tilde{v}_{\epsilon}$ with $\tilde{v}_{\epsilon} \leq \tilde{u}_c$  and 
   $\tilde{v}_{\epsilon} \nearrow \tilde{u}_c$ as $\epsilon \to 0$. 
   If there is another solution 
   $u_1$ with $u_1(x_1) <\tilde{u}_c(x_1)$ for some $x \in B_1$, then we first choose $\epsilon$ large enough so that $\tilde{v}_{\epsilon} \leq u_1$ on $\overline{B}_1$. By decreasing
   $\epsilon$, there exists $\epsilon_0$ and $x_2 \in B_1$ such that $\tilde{v}_{\epsilon_0}(x_2)=u_1(x_2)$ for $x_2 \in \{u_1>0\}\cap \{\tilde{v}_{\epsilon_0}>0\}$ 
   or $x_2 \in \partial \{u_1>0\}\cap  \partial\{\tilde{v}_{\epsilon_0}>0\}$. This would then be a contradiction. 
   By constructing a continuous family of supersolutions from above and applying an analogous argument, we determine that 
   $u_c$ is a unique solution subject to its boundary data. Since a minimizer exists and is a solution, we conclude that $u_c$ is the minimizer.

   We recall the construction of our family of subsolutions from \cite{a17}. We fix $c <c_0$ and let $\tilde{v}_{\epsilon}=\max\{r\tilde{f}_{c}(\phi) - \epsilon r^{-1/2} \tilde{g}_c(\phi),0\}$,
   see Figure 1. 
   We also fix $\epsilon$, relabel $\tilde{v}_{\epsilon}$ as $\tilde{v}$, and note that 
   $\Delta_c \tilde{v}=0$ in $\{\tilde{v}>0\}$. In this section we use computer-assisted proof to conclude that $|\nabla_c \tilde{v}|>1$ on $\{\tilde{v}=0\}$. We note that on $\{\tilde{v}=0\}$
   we have $r\tilde{f}_c(\phi)=\epsilon r^{-1/2}\tilde{g}_c(\phi)$, so that on $\{\tilde{v}=0\}$ we obtain 
   \begin{equation}   \label{e:express}
     | \nabla_c \tilde{v} |^2 = \frac{1}{1+c^2} \tilde{v}_r^2 +\frac{1}{r^2} \tilde{v}_{\phi}^2 
     = \frac{9/4}{1+c^2} \tilde{f}_c^2(\phi) + \left[\tilde{f}_c'(\phi)-\tilde{f}_c(\phi) \frac{\tilde{g}_c'(\phi)}{\tilde{g}_c(\phi)} \right]^2. 
    \end{equation}
    Note that \eqref{e:express} is the same no matter the value of $\epsilon>0$. 
   Under the change of variable $t=\cos(\phi)$ we obtain 
    \[
     | \nabla_c v |^2 = \frac{9/4}{1+c^2} f_c^2(t) + (1-t^2)\left[f_c'(t)-f_c(t) \frac{g_c'(t)}{g_c(t)} \right]^2. 
    \]
    If $t_c$ is the unique value such that $f_c(t_c)=0$, then $t_c <0$. Furthermore, if $v(r,t)=0$, then $t>t_c$. 
   Now  
     \begin{equation}  \label{e:subderiv}
      |\nabla_{c} v(r, t_c)|^2 =  (1-t_c^2)\left[f_c'(t_c)-f_c(t_c) \frac{g_c'(t_c)}{g_c(t_c)} \right]^2 =  (1-t_c^2)|f_c'(t_c)|^2 = |\tilde{f}_c'(\phi_c)|^2=1.     
     \end{equation}

 Consider now the expression 
     \begin{equation}    \label{e:deriva1}
      G_c(t):=\frac{9/4}{1+c^2} f_c^2(t) + (1-t^2)\left[f_c'(t)-f_c(t) \frac{g_c'(t)}{g_c(t)} \right]^2. 
     \end{equation}
     Notice that $G_c(t_c)=1$ and we want $G_c(t)>1$ when $t>t_c$, so that $v$ is a subsolution. It is then sufficient to show
     that $G_c'(t)>0$ for $t_c\leq t<1$. This was the key step in \cite{a17}, and a compactness argument was used to show this is the case for $c$ close to $0$. We  use computer-assisted proof to show this is true 
     whenever $c\leq \tilde c_0$ in Theorem \ref{t:subbound}. 
     
     \begin{figure}[t]     \label{f:below1}
     \includegraphics[width=60mm]{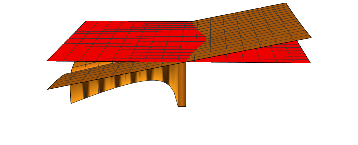}       
       \caption{The subsolution $\max\{\tilde{v},0\}$}
    \end{figure}

     \begin{theorem}  \label{t:subbound}
      If $\beta \in[0,0.58828]$, then $G_c'(t)>0$ for $t_c<t<1$, for $G_c$ as defined in \eqref{e:deriva1}.  
     \end{theorem} 
     
     Before proving Theorem \ref{t:subbound}, we establish a few helpful results.
     
     \begin{lemma}
If $\beta \in \left[\frac{3}{2},2\right]$ and $t\in [-1,1]$, then $g(t,\beta) \geq 1$. 
\label{lemma:g_is_positive}
\end{lemma}
\begin{proof}
Recall that $g(t,\beta) = f(t,-\beta/8)$. Thus 
\eq{
g(t,\beta) = \sum_{n = 0}^{\infty} a_n(-\beta/8)(1-t)^n,
}{\notag}
where $a_0(-\beta/8) = 1$, $a_1(-\beta/8) = \beta/16$, and for $n\geq 1$,
\eq{
a_{n+1}(-\beta/8) & = \frac{a_n(-\beta/8)}{2(n+1)^2}(n^2+n+\beta/8).
}{\notag}
Since $a_1(-\beta/8)>0$ for all $\beta \in \left[\frac{3}{2},2\right]$, we have by induction that $a_n(-\beta/8)>0$ for all $n\in \N$ whenever $\beta \in \left[\frac{3}{2},2\right]$. Then since $t\in [-1,1]$, we have that $g(t,\beta) \geq 1$.
\end{proof}

Recall that $f(t,\beta):= \sum_{n=0}^{\infty} a_n(\beta)(1-t)^n$ where $a_{n+1}(\beta) := \frac{a_n(\beta)}{2}\left(\frac{n^2+n-\beta)}{(n+1)^2}\right)$ for $n\geq 0$ and $a_0 = 1$. The following lemma helps bound the tail series representation of $|f(t;\beta)|$.

\begin{lemma}\label{lemma6}
If $|\beta| \leq n+1$, then
\eq{
\left|\frac{n^2+n-\beta)}{(n+1)^2}\right|  \leq 1.
}{\notag}
\end{lemma}
\begin{proof}
Assume that $|\beta|\leq n+1\in \mathbb{N}$. Then
\eq{
\left|\frac{n^2+n-\beta)}{(n+1)^2}\right| &\leq \frac{n^2+n+|\beta|}{(n+1)^2}\\
& \leq \frac{n^2+2n+1}{(n+1)^2} \\
& = 1.
}{}
\end{proof}
\begin{corollary}\label{corollary1}
If $k+1\geq |\beta|$, then for all $n\geq k+1$, $|a_n(\beta)|\leq |a_{k}(\beta)|\left(\frac{1}{2}\right)^{n-k}$.
\end{corollary}

\begin{lemma}\label{lemma7}
If $k+1\geq |\beta|$ and $|1-t|\leq r < 2$, then

\eq{
|f(t,\beta)| &\leq \left| \sum_{n = 1}^k a_n(\beta) (1-t)^{n}\right| +  \frac{|a_k(\beta)|r^{k+1}}{(2-r)},\\
|f'(t,\beta)| &\leq \left| \sum_{n = 1}^k a_n(\beta) n(1-t)^{n-1}\right| +|a_k(\beta)|\frac{2^{k-3}}{(2-r)^2},\\
|f''(t,\beta)| &\leq \left| \sum_{n = 1}^k a_n(\beta) n(n-1)(1-t)^{n-2}\right| +|a_k(\beta)|\frac{2^{k-5}}{(2-r)^3}.
}
\end{lemma}

\begin{proof}
By the corollary,  we have that
\eq{
	|f(t,\beta)|&\leq \left| \sum_{n = 1}^k a_n(\beta) (1-t)^{n}\right| +\sum_{n = k+1}^{\infty} |a_n(\beta)| r^{n}\\
	&\leq \left| \sum_{n = 1}^k a_n(\beta) (1-t)^{n}\right| + \sum_{n = k+1}^{\infty} |a_k(\beta)|\left( \frac{1}{2}\right)^{n-k}  r^{n}\\
	&= \left| \sum_{n = 1}^k a_n(\beta) (1-t)^{n}\right| + |a_k(\beta)|\left( \frac{1}{2}\right)^{-k}\sum_{n = k+1}^{\infty} \left( \frac{r}{2}\right)  ^{n}\\
	&= \left| \sum_{n = 1}^k a_n(\beta) (1-t)^{n}\right| +  \frac{|a_k(\beta)|r^{k+1}}{(2-r)}.
}{\notag}

\eq{
|f'(t,\beta)|&\leq \left| \sum_{n = 1}^k a_n(\beta) n(1-t)^{n-1}\right| +\sum_{n = k+1}^{\infty} |a_n(\beta)|n r^{n-1}\\
&\leq \left| \sum_{n = 1}^k a_n(\beta) n(1-t)^{n-1}\right| +|a_k(\beta)| \sum_{n = 0}^{\infty} \left( \frac{1}{2}\right)^{n-k} n r^{n-1}\\
&= \left| \sum_{n = 1}^k a_n(\beta) n(1-t)^{n-1}\right| +|a_k(\beta)|\left(\frac{1}{2}\right)^{3-k}\frac{1}{(2-r)^2}.
}{\notag}

\eq{
	|f''(t,\beta)|&\leq \left| \sum_{n = 1}^k a_n(\beta) n(n-1)(1-t)^{n-2}\right| +\sum_{n = k+1}^{\infty} |a_n(\beta)|n(n-1) r^{n-2}\\
	&\leq \left| \sum_{n = 1}^k a_n(\beta) n(n-1)(1-t)^{n-2}\right| +|a_k(\beta)| \sum_{n = 0}^{\infty} \left( \frac{1}{2}\right)^{n-k} n(n-1) r^{n-2}\\
	&= \left| \sum_{n = 1}^k a_n(\beta) n(n-1)(1-t)^{n-2}\right| +|a_k(\beta)|\left(\frac{1}{2}\right)^{5-k}\frac{1}{(2-r)^3}.
}{\notag}
\end{proof}

We are now ready to prove  Theorem \ref{t:subbound}. 	
	\begin{cproof}
Recall that 
\eq{
G_c(t):= \sigma f^2(t;\beta_c) + (1-t^2)\left[f'(t;\beta_c)-\frac{f(t;\beta_c)g'(t;\beta_c)}{g(t;\beta_c)}\right]^2,
}
where $\sigma:= (9/4)/(1+c^2)$, $\beta_c = 2/(1+c^2)$, and $'=\frac{\partial}{\partial t}$. Suppressing the function arguments, we have
\eq{
G_c'(t) = 2\sigma ff'-2t(f'-fg'/g)^2+2(1-t^2)(f'-fg'/g)(f''-f'g'/g-fg''/g+f(g')^2/g^2).
}{}
In Lemma \ref{lemma:g_is_positive}, we showed that $g > 0$ for $t_c\leq t\leq 1$, and so it suffices to show that $g^3G'(t) > 0$. That is, it suffices to show that
\eqn{
0<& 2\sigma f f'g^3-2t(f'g-fg')^2g+ \\
&2(1-t^2)(f'g-fg')(f''g^2-f'g'g-fg''g+f(g')^2).
}{eq:G_fun:modified}
To accomplish our goal, we will interpolate the RHS of \eqref{eq:G_fun:modified} together with analytic interpolation error bounds, as described in \cite{Barker,Barker2014}. We interpolate $g^3G'$ on the domain $(t,\beta)\in [-0.204,1]\times [-0.161,2]$. After a linear change of coordinates that maps 
the domain to $[-1,1]\times [-1,1]$, we form ellipses $E_t$ and $E_{\beta}$ in the complex plane,  where the stadiums are as defined in \cite{RW}, for the variables $t$ and $\beta$ using $\rho_{t} = 2.9$ and $\rho_{\beta} = 30$. In order to use analytic interpolation error bounds, we must show $g^3G'(\cdot,\beta)$ is analytic inside and on the ellipse $E_t$ for all $\beta\in[-1,1]$, and that $g^3G'(t,\cdot)$ is analytic inside and on the ellipse $E_{\beta}$ for all $t\in [-1,1]$. Note that $g^3G'$ will be analytic as long as $f$, $f'$, $f''$, $g$, $g'$, $g''$ are analytic, which they are as given by our rigorous computations related to the Lemmata \ref{lemma6}-\ref{lemma7} and Corollary \ref{corollary1}, with $k =  \lceil\frac{a+b}{2} +\frac{b-a}{2}(\rho_{\beta}+\frac{1}{\rho_{\beta}})  \rceil - 1$, which ensures $k+1 \ge |\beta|$, where $|\beta| \leq \frac{a+b}{2} +\frac{b-a}{2}(\rho+\frac{1}{\rho})$. We record the bounds we obtain, $|f|,|g|\leq M_1$, $|f'|, |g'|\leq M_2$, $|f''|,|g''|\leq M_3$, in Table 1. Applying standard bounds, we have
\eq{
|g^3(t)G'(t)| &\leq 2B_{\sigma}M_1M_2M_1^3+8B_{E_{t}}M_1^3M_2^2\\&+ 4(1+B_{E_{t}}^2)(M_1M_2)(M_3M_1^2+M_1^2M_2+M_1^2M_3+M_1M_2^2),
}{}
where $B_{\sigma} = \frac{9}{8}B_{E_{{\beta}}}$, and $B_{E_{{\beta}}}$ is a bound on the modulus of elements in $E_{\beta}$, and $B_{E_t}$ is a bound on the modulus of elements in $E_t$.

We break the interval $[0,0.58828]$ up into several sub intervals, and then for each of those subintervals, we get a lower bound on the root of $f$ using our rigorous interpolation of $f$. Then we evaluate, using the method described in \cite{Barker,Barker2014},  our rigorous interpolation of $g^3G'$ on a grid of subintervals in $t$ and $\beta$ to confirm that $g^3G' \geq 0$ on that sub-grid. The code for this rigorous computation is available on GitHub at \url{https://github.com/bhbarker/rigorous_computation/tree/main/minimizers_free_boundary_cone}.

\end{cproof}

\begin{center}
\begin{table}\label{tb:bounds1b}
\begin{tabular}{| c ||c| c|c| }
\hline
Function & $|f|$, $|g|$ & $|f'|$, $|g'|$ & $|f''|$, $|g''|$\\
\hline
Bound &       12402 &      146200 &    553380\\
\hline
\end{tabular}
\caption{Table demonstrating the bounds on $f$ and $g$ and their first two derivatives on an ellipse.}
\end{table}
\end{center}

 \section{A continuous family of supersolutions}  \label{s:supersolution} 
 
  \begin{figure}[h]     \label{f:above}
     \includegraphics[width=60mm]{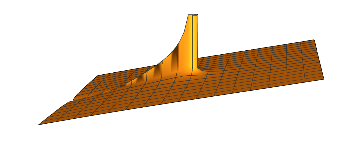}       
       \caption{  $\tilde{v}_c=r\tilde{f}_{c}(\phi) + \epsilon r^{-1/2} \tilde{g}_c(\phi)$. }
    \end{figure} 
 
   In order to construct a family of supersolutions we begin with ideas similar to those in Section \ref{s:subsolutions}. In this section we consider 
   $\tilde{v}_c=r\tilde{f}_{c}(\phi) + \epsilon r^{-1/2} \tilde{g}_c(\phi)$. We note that both $f_c$ and $g_c$ have regular singular points at $t=-1$ and also the same 
   asymptotic expansion around $t=-1$. Therefore, $\tilde{v}_c$ will be positive in a neighborhood of the origin (and increase without bound close to the origin).  Therefore, 
   in the distributional sense $\Delta_c \tilde{v}_c \leq 0$ in $\{\tilde{v}_c >0\}$. 
    
   Figure 2 illustrates why the construction of a supersolution will be more difficult. 
   Although it is true that $\Delta_c \tilde{v}_c \leq 0$ in $\{\tilde{v}_c>0\}$, the second condition $|\nabla_c \tilde{v}_c | <1$ on $\partial \{\tilde{v}_c>0\}$
   will not be true because of the singularity 
   at the origin. However, from Theorem \ref{t:subbound} the second condition will be true whenever $\phi$ is close to $\phi_c$. We will utilize this fact to modify
   $\tilde{v}_c$ and construct a supersolution to our free boundary problem. 
   To remove the portion of $\tilde{v}_c$ that contains the singularity, 
   we construct a $c$-harmonic function with boundary values close to that of $\tilde{v}_c(r,\phi)$.  If we have a $c$-harmonic function of the form $r^{\alpha}\tilde{h}(\phi)$, 
   then from \eqref{e:legalpha} we have under the change of variable that $h(t)$ is a Legendre function. Since we need $r^{\alpha}\tilde{h}(\phi)$ to be $c$-harmonic
   in all of $B_1$ it follows that $h(t)$ must be a Legendre polynomial and that 
   \[
    \frac{\alpha(\alpha+1)}{1+c^2} = n(n+1) \quad \text{ for some } n \in \mathbb{N}. 
   \]
   To denote the dependence on $n$ we write $\alpha_n$, and solving for $\alpha_n$ we obtain
   \[
    \alpha_n = \frac{-1+ \sqrt{1+4n(n+1)(1+c^2)}}{2}. 
   \]
   Our goal is to approximate $f_c(t) + \epsilon g_c(t)$ with Legendre functions $h_n(t)$. If $f_c(t) + \epsilon g_c(t) \approx \sum_{n=0}^N a_n h_n$ with $h_n$ a Legendre polynomial of 
   degree $n$, then $w(r,t)=\sum_{n=0} a_n r^{\alpha_n} h_n(t)$ will be a $c$-harmonic function. We note that $h_n$ are independent of $c$, but the coefficients 
   $a_n$ and the powers $\alpha_n$ will vary with $c$. 
   By appropriately choosing $a_n$,  we will obtain that $p_c(r,t):=\max\{\min\{v_c,w_c\},0\}$ will be 
   a supersolution to our free boundary problem. This is the content of the next Lemma.

   \begin{lemma}   \label{l:above}
    Let $0 \leq c < c_0$. There exists $\epsilon>0$ and finitely many coefficients $a_n$ such that if $w_c(r,t)=\sum_{n=0} a_n r^{\alpha_n} h_n(t)$, and if 
    $p_c(r,t)=\max\{\min\{v_c,w_c\},0\}$, then the following conditions hold:
    \[
     \begin{aligned}
     &(1) \quad w_c(1,t)>f_c(t) + \epsilon g_c(t) \quad  \text{ on } \{w_c(1,t)>0\}. \\
     &(2) \quad \Delta_c \tilde{p}_c(r,\phi) \leq 0 \text{ in } \ \{\tilde{p}_c(r,\phi)>0\} \ \text{ in the distributional sense.} \\
     &(3) \lim_{\rho \to 0} \sup_{B_{\rho}(x)} |\nabla_c \tilde{p}_c(x)| <1 \quad \text{ for any } x \in \partial\{\tilde{p}_c>0\}.   
     \end{aligned}
    \]
   \end{lemma}
   
   \begin{remark}
    property $(2)$ and $(3)$ will guarantee that $p_c$ is a supersolution to the free boundary problem. 
    In regards to property $(2)$: the minimum of two supersolutions is a supersolution in the distributional sense 
    (since the minimum may not be differentiable). A supersolution in the distributional sense will still satisfy the comparison principle. Indeed, a straightforward calculation shows 
    that the minimum of two classical supersolutions will satisfy the comparison principle. In regards to property $(3)$: the interface $\partial \{p_c>0\}$ may not be differentiable. However,
     property $(3)$ will be sufficient for a comparison principle to establish that $p_c$ is a supersolution to the free boundary problem. With property $(2)$ and $(3)$ one may wonder why
     property $(1)$ is necessary or emphasized. In Lemma \ref{l:above} we have the $\epsilon$ fixed and hence a particular supersolution. In order to construct the continuous family of 
     supersolutions for Theorem \ref{t:main}, we will need to utilize property $(1)$. 
   \end{remark}

  \begin{proof}
   
   To rigorously verify conditions (1) and (3) of the lemma, we use rigorous computation as outlined in the following algorithm. 
   
 \begin{enumerate}
 \item The coefficients we use to construct $w_c$ are given in Table 2, and were obtained in a guess and check manner by visualizing the conditions (1) and (3) of the lemma for specific values of $c$. We divide the interval $[0,0.43]$ into smaller intervals for which the said coefficients suffice to establish conditions (1) and (3) of the lemma. These subintervals are recorded in Table 2.
 \item We interpolate $f_c$ with rigorous error bounds for $t\in[ -0.7,1]$ and $\beta \in [1.47,2]$. The details of interpolation are similar to those described in Section \ref{s:subsolutions}.
 \item We interpolate $f_c$ with rigorous error bounds for $t\in[-.7,1]$ and $\beta \in [-1/4,2]$. Recall that $g_c(t,\beta) = f_c(t,-\beta/8)$. We then use the interpolation of $f_c$ to compute $g_c$.
 \item To verify condition (1) of the lemma, we compute $w_c$ with $r = 1$ on a 10,000 by 1,000 grid of intervals with $(t,\beta) \in[-1,1]\times [\beta_a,\beta_b]$, where $[\beta_a,\beta_b]$ corresponds to an interval $[c_a,c_b]$ listed in the last column of Table 2. Note that the code implements vectorization to reduce the amount of times the computer rounding mode must be changed. Next, we determine where $w_c>0$ and verify that condition (1) of the lemma holds whenever this is the case.
 \item Next, we use 10 iterations of Newton's method with point intervals to get a close approximation of where $w_c = f_c+\varepsilon g_c = 0$, referred to as the cross point. Next, we make an interval of radius 1e-3 centered at this approximation of the cross point and iterate the interval Newton method 10 times to obtain a rigorous enclosure of where $w_c = f_c+\varepsilon g_c = 0$. 
 \item Next, we check that condition (3) of the lemma holds on the interval enclosure of the cross point found in the previous step.
 \item Finally, if $(x_0,y_0)$ is the cross point, we compute where $w_c = 0$ on a grid of size 30 by 30 with $x\in [-1,-\sqrt{x_0^2+y_0^2}]$ and $y\in[0,y_0]$. Wherever $w_c = 0$, we verify that $\left | \nabla_c w_c\right | < 1$.
  \item Note that the code is available at \url{https://github.com/bhbarker/rigorous_computation/tree/main/minimizers_free_boundary_cone}.
  \end{enumerate}
   
    \begin{center}
    \begin{table}
    \setlength{\tabcolsep}{10pt} 
\renewcommand{\arraystretch}{1.5} 
     \begin{tabular}{|c||c|c|c|c|c|c|}
\hline Range of $c$ & $\epsilon$ & $a_0$ & $a_1$ & $a_2$ & $a_3$ & $c$-intervals   \\
\hline \hline   
$0\leq c \leq 0.3$ & $0.2$ & $0.999$ & $0.31$ & $0$ & $0$ & $\left[\frac{j}{10},\frac{j+1}{10}\right]$ for $j=0,1,2$.   \\ 
\hline
$ 0.3\leq c \leq 0.41$ & $0.1$ & $0.19$ & $1.005$ & $-0.09$ & $0.03$& $\left[\frac{30+j}{10},\frac{30+j+1}{10}\right]$ for $j=0,1,...,10$.   \\
\hline
$0.41 \leq c \leq 0.42$ & $0.1$ &    $0.193$    & $1.005$    & $-0.09$   & $0.03$&  [0.41,0.42]   \\
\hline
$0.42 \leq c \leq 0.43$ & $0.1$ &    $0.196$    & $1.005$    & $-0.09$   & $0.03$ & [0.42,0.43]   \\
\hline
\end{tabular}
\label{tb:coeffs}
\caption{We record the coefficients $\alpha_n$ used in the definition of the function $w$ and $\varepsilon$ that appears in the sum $f_c+\varepsilon g_c$ for various intervals of $c$. }
\end{table}
\end{center}

This concludes the proof.
  \end{proof}

  \begin{remark}
   The value of $c=0.43$ is not optimal using the method we have shown; however, the optimal value (with this method) is not too much higher than $0.43$. One may guess that 
   by simply using a higher order and better approximation for $f_c + \epsilon g_c$ on $\partial B_1$ this method will work for higher values of $c$. This is not the case, and most likely one cannot consider (for the role of $w_c$) the positivity set of a function that is $c$-harmonic in all of $B_1$. 
     \end{remark}

  The intent of this next Lemma is to use a continuous one-parameter family of supersolutions satisfying \eqref{e:sups} to 
  conclude that any solution to \eqref{e:phase1} lies below $\tilde{p}_c(r,\phi)=p_c(r, \cos \phi)$. The coefficients in this next Lemma are easily obtained and apply 
  to much larger intervals for values of $c$. The inequalities are are strict with large differences (often greater than $0.1$); consequently, we did not apply the computer assisted proof to this next Lemma. 
  
   \begin{lemma}  \label{l:belowp}
      Let $\tilde{p}_c(r,\phi)=p_c(r, \cos \phi)$ with $p_c(r,t)$ given as in Lemma \ref{l:above}. If $\psi$ is a solution to \eqref{e:phase1}, 
      with $\psi = \tilde{u_c}$ on $\partial B_1$, then $\psi \leq \tilde{p}_c$. 
   \end{lemma}
   
   \begin{proof} 
    To simplify the presentation we handle the different ranges of $c$ separately. For each fixed range, only the powers $\alpha_n$ will vary with $c$. The
    coefficients $a_n$ are fixed for the particular range of $c$. We first assume $0 \leq c \leq 0.3$. If
    \[
     q_s(r,t):= \max\{0,s a_0 + 0.8 a_1 r^{\alpha_1} h_1(t)\},
    \]
    then we compute (which in this instance can be done by hand) that $|\nabla_c q_s| < 1$ on $B_1$. Furthermore, for $s= 0.22$ we have that 
    $\text{supp}\  q_{0.22}(r,t) \subset \text{supp} \ p_c(r,t)$. Also, $u_c(1,t) \leq q_{0.22}(1,t)$. Using the comparison principle, and the one-parameter family $\tilde{q}_s$ of supersolutions
    satisfying \eqref{e:sups}, 
    it follows that if 
    $\psi$ satisfies \eqref{e:phase1} and $\psi = \tilde{u}_c$ on $\partial B_1$, then $\text{supp} \ \psi \subset \text{supp} \ \tilde{p}_c$ with $\psi \leq \tilde{q}_{0.22} \leq \tilde{p}_c$ on $\partial B_1$. It follows by the comparison principle for $\Delta_c$, that $\psi \leq \tilde{p}_c$ on $B_1$.  
    
    We now consider the range $0.3 \leq c \leq 0.4$, and 
    \[
     q_{s}(r,t):= \max\{0, s_1 a_0 + s_2 \sum_{n=1}^3 a_n h_n(t)\}. 
    \]
    with 
    \[
     s_1(s)=
      \begin{cases}
       s-1 &\text{ for } 4 \leq s \\
       3 & \text{ for } 3\leq s \leq 4 \\
       1.5s-1.5 &\text{ for } 2 \leq s \leq 3 \\
       1.5 &\text{ for } 1 \leq s \leq 2\\
       0.65s+0.85 & \text{ for } 0\leq s \leq 1 
      \end{cases}
    \qquad 
     s_s(s)=
      \begin{cases}
       0.7 &\text{ for } s \leq 4 \\
       -0.1s+1.1 & \text{ for } 3\leq s \leq 4 \\
       0.8 &\text{ for } 2 \leq s \leq 3 \\
       -.15s+1.1 &\text{ for } 1 \leq s \leq 2\\
       0.95 & \text{ for } 0\leq s \leq 1 
      \end{cases}
    \]
    The values of $s_1, s_2$ are chosen for simplicity in computation rather than for presentation. For
     $s=0$, we have that $q_0 \leq p_c$ and also $q_0(1,t) \geq u_c(1,t)$. It then follows as before that for any $\Phi$ a solution to \eqref{e:phase1} with 
     $\Phi=\tilde{u}_c$ on $\partial B_1$ that $\Phi \leq \tilde{p}_c$. 
         
   \end{proof}

    \begin{theorem} \label{t:main}
      Let $0 \leq c \leq 0.43$. Then $\tilde{u}_c$ is a unqiue minimizer of the functional $J_{c}$ subject to Dirichlet boundary conditions, i.e. if $\psi - u \in W_0^{1,2}(B_1)$, 
      then 
      \[
       J(u_c) \leq J(\psi), 
      \]
      and consequently the free boundary of a minimizer of $J_c$ can pass through the vertex of the cone. 
    \end{theorem}

    \begin{proof}
     We fix $c \leq 0.43$. From Lemma \ref{l:belowp} we have that any solution $\psi$ to \eqref{e:phase1} with $\psi - u \in W_0^{1,2}(B_1)$ one has that 
     $\psi \leq \tilde{p}_c$. 
     Following the ideas in \cite{a17}, we now construct a continuous one-parameter family of 
     supersolutions $\tilde{p}_{c,\tau}$ such that $\tilde{p}_{c,\tau} \searrow \tilde{u}_c$. From the strict comparison principle, it then follows that any solution to 
     \eqref{e:phase1} lies below $\tilde{u}_c$. From Section \ref{s:subsolutions} we also have that any solution to \eqref{e:phase1} lies above $\tilde{u}_c$. It then 
     follows that $\tilde{u}_c$ is a minimizer and is also a unique minimizer. We now proceed with the construction of the one-parameter family of supersolutions.


     Let $t_{\epsilon}$ be the value such that $v_c(1,t_\epsilon)=0$.  We consider $\tilde{v}_{c,\tau}:=r\tilde{f}_c(\phi)+\tau r^{-1/2}\tilde{g}_c(\phi)$ with $0<\tau \leq \epsilon $ with $\epsilon$ given in Lemma \ref{l:above}. 
     We let $\rho:=(\tau/\epsilon)^{3/2}$. We will redefine $\tilde{v}_{c,\tau}$ on the ball $B_{\rho}$. In order to motivate this we consider the scaling 
     \[
     \tilde{v}^{\rho}_{c,\tau}(r,\phi):= \frac{v_{c,\tau}(\rho r,\phi)}{\rho}
     \]
     which will take 
     $B_1 \to B_{1/\rho}$.      
     Now under this scaling we obtain 
     \[
      \tilde{v}^{\rho}_{c,\tau}(r,\phi)=[r\tilde{f}_c(\phi)+\tau \rho^{-3/2}r^{-1/2}\tilde{g}(\phi)]= \tilde{v}_{c,\epsilon}(r,\phi).
     \] 
     Outside of $B_1$ we will have that the rescaled function $\tilde{v}^{\rho}_{c,\tau}$ is a supersolution to the free boundary since if $v^{\rho}_{c,\tau}(r_0,t_0)=0$ with $r \geq 1$, then 
     $t_0 \geq t_{\epsilon}$, so that if $t_0 =\cos \phi_0$, then $|\nabla_c \tilde{v}^{\rho}_{c,\tau}(r_o,\phi_0)|<1$ (this follows from Theorem \ref{t:subbound}).  We now let 
     \[
     \overline{p}_{c,\tau}(r,\phi):=
     \begin{cases}
      \tilde{v}_{c,\epsilon}(r,\phi) \text{ if } r\geq 1 \\
      \tilde{p}(r,\phi) \text{ if } r<1. 
     \end{cases}
     \] 
     From property $(1)$ in Lemma \ref{l:above}, we have that $\overline{p}_{c,\tau}(r,\phi)=v_{c,\epsilon}(r,\phi)$ in a neighborhood of $(1,\phi)$ whenever \
     $\tilde{v}_{c,\epsilon}(1,\phi)>0$. Then 
     $\overline{p}_{c,\tau}$ is a supersolution to the free boundary problem in $B_{1/\rho}$ even across $\partial B_1$. Rescaling back to $B_1$ by defining 
     \[
      \tilde{p}_{c,\tau}(r,\phi):=\rho\overline{p}_{c,\tau}(r/\rho,\phi)
     \]
     we have a continuous family of supersolutions $p_{c,\tau} \searrow v_c$. This concludes the proof. 
    \end{proof}

\appendix

\section{Computational environment} Rigorous computations were run on a MacBook Pro (Retina, 15-inch, Mid 2015) with a 2.8 GHz Quad-Core Intel Core i7 and 16 GB 1600 MHz DDR3, running version 10.15.7 of Catalina. The computational environment consisted of MatLab 2019b and INTLAB version 11. All code can be obtained at \url{https://github.com/bhbarker/rigorous_computation/tree/main/minimizers_free_boundary_cone}.

\bibliographystyle{amsplain}
\bibliography{refccone} 

\end{document}